\documentclass[a4paper, 10pt, conference]{ieeeconf}      

\IEEEoverridecommandlockouts                              
\usepackage{amsmath,bm} 
\usepackage{amssymb,mathrsfs,color,nameref}  

\usepackage{graphicx} 
\usepackage{cite,supertabular}

\newtheorem{theorem}{Theorem}[section]
\newtheorem{lemma}[theorem]{Lemma}

\newtheorem{definition}[theorem]{Definition}
\newtheorem{remark}[theorem]{Remark}

\newtheorem{example}[theorem]{Example}

\newenvironment{proofof}[1]{\smallskip\noindent{\emph{Proof~of~#1:}}%
  \hspace{1pt}}{\hspace{-5pt}{\nobreak\quad\nobreak\hfill\nobreak%
    $\square$\vspace{2pt}\par}\smallskip\goodbreak}

\newcommand{\spt}{\mathop{\rm spt}}

\newcommand{\co}{\mathop{\rm co}}

\newcommand{\id}{\mathop{\rm id}}
\newcommand{\dom}{\mathop{\rm dom}}
\newcommand{\diam}{\mathop{\rm diam}}
\newcommand{\lip}{\mathop{\rm Lip}}

\newcommand{\Lip}{\mathrm{Lip}}

\begin{document}


\title{Stabilization via localized controls in nonlocal models of crowd dynamics}

\author{
Nikolay Pogodaev$^{a}$,
Francesco Rossi$^{b}$
\thanks{$^{a}$ Dipartimento di Matematica ``Tullio Levi-Civita'', Università degli Studi di Padova; \texttt{pogodaev@math.unipd.it}}
\thanks{$^{b}$  Dipartimento di Culture del Progetto, Universtà Iuav di Venezia; \texttt{francesco.rossi@iuav.it} He is a member of GNAMPA (INdAM).}
}

\maketitle
\thispagestyle{empty}

\begin{abstract}

  We consider a control system driven by a nonlocal continuity equation.
  Admissible controls are Lipschitz vector fields acting inside a fixed open set.
  We demonstrate that small perturbations of the initial measure, traced along Wasserstein geodesics, may be neutralized by admissible controls.
  More specifically, initial perturbations of order \( \varepsilon \) can be reduced to order \( \varepsilon^{1+\kappa} \), where \( \kappa \) is a positive constant.
\end{abstract}


\section{Introduction}
\label{sec:introduction}

In recent years, the study of systems describing a crowd of interacting autonomous agents has attracted significant interest from the mathematical and control communities.
A better understanding of such interaction phenomena can have a strong impact on several key applications, such as road traffic and egress problems for pedestrians.
For a few reviews on this topic, we refer to~\cite{axelrod,CPTbook,helbing,SepulchreReview}.

Beside the description of interactions, it is now relevant to study problems of crowd control, i.e., of controlling such systems by acting on few agents, or within a small subset of the configuration space.
Basic problems for such models include controllability (i.e., reaching a desired configuration), optimal control (i.e., the minimization of a given functional) and stabilization (i.e., counteract perturbations to stay around a given configuration/trajectory).
Many results in these directions can be found in \cite{blaq,DMRcontrol,FS,carmona,bullo,achdou1,CHM,fornasierMeanfieldOptimalControl2018,chertovskihOptimalControlNonlocal2023,Ciampa2021185}.

There is a wealth of models available to describe crowds.
Here, we choose one of the most popular in the control community: the continuity equation with non-local velocity.
Here, the standard setup for the control system is as follows:
\begin{equation}
  \label{eq:conteq}
\partial_t\mu_t+\nabla_x\cdot \left(\left(V_t(x,\mu_t)+u_t(x)\right)\mu_t\right) = 0,
\end{equation}
where \( u_t \) is an external vector playing the role of control.
As usual, Equation~\eqref{eq:conteq} should be understood in the weak (distributional) sense.

The state space of~\eqref{eq:conteq} consists of all compactly supported probability measures on \( \mathbb{R}^d \) and is denoted by \( \mathcal{P}_c(\mathbb{R}^d) \).
We equip it with the quadratic Wasserstein distance \( \mathcal{W}_2 \).
This distance is closely related to the optimal transportation problem, see Section \ref{sec:wass}. As an important consequence, it turns \( \mathcal{P}_c(\mathbb{R}^d) \) into a geodesic space: any pair of points \( \varrho_0, \varrho_1 \) can be joined by an absolutely continuous curve of length \( \mathcal{W}_2(\varrho_0,\varrho_1) \), called a \emph{geodesic}.

To complete the description of the control system, we need to specify a class of admissible controls.
In this paper, we choose the class of \emph{localized Lipschitz} controls. By \emph{Lipschitz control} we mean any time-dependent vector field \( u \) such that
\begin{equation}
  \label{eq:CL}
u\in L^\infty\left([0,T];\Lip(\mathbb{R}^d;\mathbb{R}^{d})\right),
\end{equation}
where \( \Lip(\mathbb{R}^d;\mathbb{R}^{d}) \) is the Banach space of Lipschitz maps \( f\colon \mathbb{R}^d\to \mathbb{R}^d \).
We say that a Lipschitz control \( u \) is \emph{localized} if it acts only on a certain subset of the \( x \)-space.
Formally, given a non-empty open subset \( \omega \) of \( \mathbb{R}^d \), we require that \( \spt( u_t)\subset \omega \) for all \( t\in [0,T] \).
In what follows, such controls will be also called \emph{admissible}.
The corresponding solutions \( \mu^u_t \) of~\eqref{eq:conteq} will be called \emph{admissible trajectories}.

The regularity assumption~\eqref{eq:CL} is rather natural.
It guarantees~\cite{ambrosioTransportEquationCauchy2008} that the continuity equation
\begin{equation}
  \label{eq:local}
\partial_t\mu_t+\nabla_x\cdot \left(u_t(x)\mu_t\right) = 0
\end{equation}
has a \emph{unique} weak solution for any initial data \( \mu_0=\varrho_0 \).
If~\eqref{eq:CL} is violated, then weak solutions are typically nonunique.
The same holds for~\eqref{eq:conteq} under the standard regularity hypothesis on \( V \) detailed below in Assumption \( (A_1) \).

The stabilization problem that we study here can be roughly stated in the following way: Let \( \mu_t \) be the trajectory of~\eqref{eq:conteq} corresponding to an initial measure \( \varrho_0 \) and the zero control \( u=0 \); we call \( \mu_t \) the \emph{reference trajectory}.
Suppose that we slightly perturb \( \varrho_0 \), getting a new measure \( \varrho \).
Can we steer the system back towards the reference trajectory by using only localized Lipschitz controls? This intuitive idea of stabilization is made precise in the following definition.

\begin{definition}
  \label{def:stab}
    We say that a set $A\subset \mathcal{P}_c(\mathbb{R}^d)$ is \( \kappa \)-\emph{stabilized around the reference trajectory \( \mu_t \)} of \eqref{eq:conteq} if there exists \( C>0 \) such that for any \( \varepsilon>0 \) and $\varrho\in A$ with \( \mathcal{W}_2(\varrho,\varrho_0)< \varepsilon \) one can find an admissible $u$ such that the corresponding trajectory $\mu_t^u$ of \eqref{eq:conteq} starting from $\varrho$ satisfies
  \begin{equation}
    \label{eq:stabilization}
     \mathcal{W}_2\left(\mu^u_{T},\mu_{T}\right)< C\varepsilon^{1+\kappa}.
   \end{equation}
   In other words, whatever \( \varepsilon \) we choose, any point of the set \( A\cap{\bf B}_{\varepsilon}(\varrho_0) \) can be steered into the ball \( {\bf B}_{C\varepsilon^{1+\kappa}}(\mu_{T}) \) by an admissible control, see~Fig~\ref{fig:paths}.
   Here \( {\bf B}_\varepsilon(\varrho_0) \) denotes the open Wasserstein ball of radius \( \varepsilon \) centered at \( \varrho_0 \).
 \end{definition}

\begin{figure}[htb]
  \begin{center}
    \includegraphics[width=0.45\textwidth]{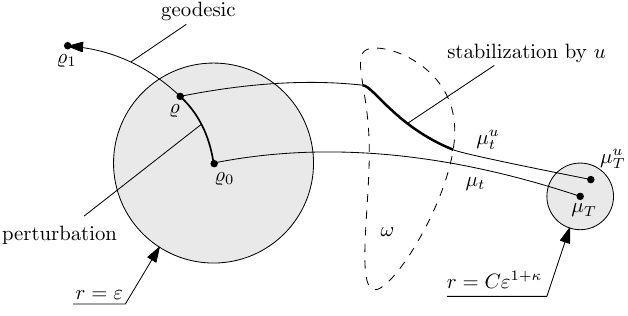}
    \caption{
      The concept of stabilization: we perturb \( \varrho_0 \) to a new measure \( \varrho \). We reduce the perturbation rate with a Lipschitz control localized in \( \omega \).
    }
\label{fig:paths}
\end{center}
\end{figure}

  Now, we would like to understand which sets \( A \) of initial measures \emph{are} \( \kappa \)-stabilized.
  A fruitful idea, as we shall see, is to generate such sets by specifying admissible perturbations of \( \varrho_0 \).
  More specifically, we define the set of starting points
\[
S^{\Pi}_r(\varrho_0) := \left\{\varrho_{\varepsilon}\text{ where } \varrho\in \Pi(\varrho_0),\; \varepsilon\in [0,r]\right\},\; r\in [0,1],
\]
where \( \Pi(\varrho_0) \) is an arbitrary set of curves \( \varrho\colon [0,1]\to \mathcal{P}_c(\mathbb{R}^d) \) with the starting point \( \varrho_0 \).
In what follows, the set \( \Pi(\varrho_0) \) will be called the \emph{set of initial perturbations}.

Given \( \Pi(\varrho_0) \), we may ask whether the corresponding starting set \( S^{\Pi}_r(\varrho_0) \) is \( \kappa \)-stabilized for some \( \kappa,r>0 \)?
Example 4.8 from~\cite{PR2024}, which is just an elaborate version of Example~\ref{ex:nostabilization}, shows that \( \Pi(\varrho_0) \) cannot be composed of \emph{all} absolutely continuous curves originated at \( \varrho_0 \).
Hence, we must look for smaller sets of initial perturbations.

In the present paper, we study a very natural class of initial perturbations composed of Wasserstein geodesics starting at \( \varrho_0 \).
More specifically, we fix a compact convex set \( \Omega \subset \mathbb{R}^d\) and \( \varrho_0\in \mathcal{P}_{ac}(\Omega) \), where \( \mathcal{P}_{ac}(\Omega) \) denotes the set of all \emph{absolutely continuous} probability measures on \( \Omega \).
It is known that \( \varrho_0 \) can be joined with any other measure \( \varrho_1\in \mathcal{P}_{ac}(\Omega) \) by a unique constant speed geodesic \( \gamma_{\varrho_0,\rho_1} \) (see Section~\ref{sec:wass} for details).
Then we take
\begin{equation}
  \label{eq:Pi}
  \Pi(\varrho_0):= \left\{ \gamma_{\varrho_0,\varrho_1} \text{ where } \varrho_1\in \mathcal{P}_{ac}(\Omega)\right\}.
\end{equation}
We will show that with this choice of the perturbation class the starting set \( S^{\Pi}_r(\varrho_0) \) is \( \kappa \)-stabilizable for some \( \kappa,r>0 \).
The precise formulation of the result will be given below.

In~\cite{PR2024}, we studied the same stabilization problem as presented here, but for a different perturbation class:
\[
  \Pi(\varrho_0)\!:=\! \left\{
    \begin{matrix}
    \varepsilon\mapsto \Phi^v_{0,\varepsilon\sharp}\varrho_0\,\colon v\in L^{\infty}\left([0,1]; \Lip(\mathbb{R}^d;\mathbb{R}^d)\right),\\ \|v_\varepsilon\|_{\infty}\le 1,\,\lip(v_\varepsilon)\le1,\, \text{for a.e. }\varepsilon\in[0,1]
    \end{matrix}
    \right\},
\]
where \( \Phi^v \) denotes the flow of the time-dependent vector field \( v \), \( \|f\|_{\infty} \) the usual \( \sup  \)-norm and \( \lip(f) \) the smallest Lipschitz constant of \( f\).
We showed that the corresponding starting set \( S^{\Pi}_{r}(\varrho_0) \) is \( \kappa \)-stabilized for some \( \kappa,r>0 \).
In other words, small initial perturbations generated by Lipschitz controls can be neutralized by \emph{localized} Lipschitz controls. 
A related problem of controllability by localized Lipschitz controls, was considered for Equation~\eqref{eq:conteq} in~\cite{DMRcontrol,DMRmin}.
However, these papers dealt only with the case \( V_t(x,\mu)=V(x) \).

If the perturbation class in~\cite{PR2024} is relevant from the control theoretic viewpoint, the class~\eqref{eq:Pi} considered here is important for geometrical reasons.
Indeed, it is known that the tangent space at \( \varrho_0 \) to \( \mathcal{P}_c(\mathbb{R}^d) \) can be seen as a completion of the set of all geodesics starting from \( \varrho_0 \), as is shown in~\cite{gigli2011inverse}.

\subsection{Assumptions and main result}

We first  precisely state the assumptions imposed on the vector field \( V \), the initial measure \(\varrho_0\) and the action set \( \omega \).
\textbf{Assumption \((A_1)\):} The nonlocal time-dependent vector field  \( V\colon [0,T]\times \mathbb{R}^d \times \mathcal{P}_c(\mathbb{R}^{d})\to \mathbb{R}^d \) satisfies
\begin{itemize}
	\item \( V \) is measurable in \( t \);
	\item \( V \) is bounded and Lipschitz in \( x \) and \( \mu \), i.e., there exists \( M > 0 \) such that
	    \begin{gather*}
        \left|V_t(x,\mu)\right|\le M,\\
        \left|V_{t}(x,\mu)-V_{t}(x',\mu')\right|\le M\left(|x-x'|+ \mathcal{W}_{2}(\mu,\mu')\right),
	  \end{gather*}
		for all \(x,x'\in \mathbb{R}^d \), \( \mu,\mu'\in \mathcal{P}_c(\mathbb{R}^d) \), \(t\in [0,T] \).
\end{itemize}

This assumption, which dates back at least to~\cite{piccoliTransportEquationNonlocal2013}, ensures existence and uniqueness of solutions of the associated continuity equation \eqref{eq:conteq}.
Given an initial condition $\varrho_0$, we denote by $\mu_t$ the solution of~\eqref{eq:conteq} corresponding to the initial condition \( \mu_0=\varrho_0 \) and the zero control \( u=0 \) (i.e., the reference trajectory).
We denote by \( \Phi^{V} \) the flow of the corresponding \emph{nonautonomous} vector field \( (t,x)\mapsto V_t(x,\mu_{t}) \).

We now set the geometric condition on $\omega$, the set on which the control action is localized.

\textbf{Assumption \((A_2)\):} Given the initial datum of the reference trajectory $\varrho_0$, for any \( x\in \spt (\varrho_0) \) there exists \( t\in (0,T) \) such that \( \Phi^V_{0,t}(x)\in \omega \).

This assumption is illustrated in Figure~\ref{f-geometric}. As stated above, this condition is very natural, since it requires the reference trajectory to cross the set in which the localized control operates.
We highlight that this condition is a property of the reference trajectory and not of its perturbations.

\begin{figure}[htb]
  \begin{center}
    \includegraphics[width=0.34\textwidth]{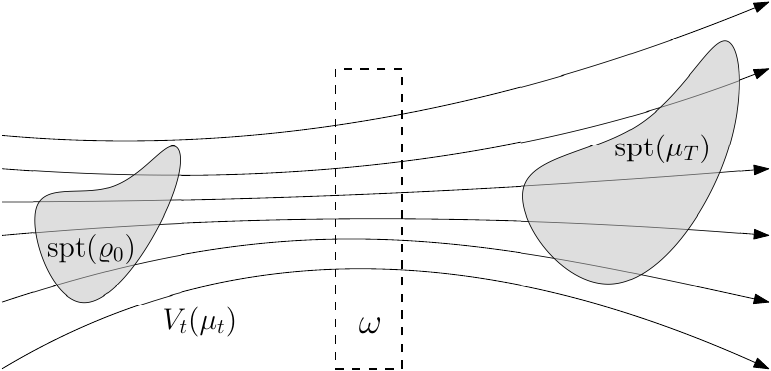}
\caption{Geometric Condition \( (A_2) \):
any trajectory of \( V_t(\mu_t) \) issuing from \( \spt(\varrho_0) \) crosses \( \omega \) by the time \( T \).}
\label{f-geometric}
\end{center}
\end{figure}

Now we are ready to state the main result of the paper.

\begin{theorem}
  \label{thm:main}
  Let a nonlocal vector field \( V \), an open set \( \omega\subset \mathbb{R}^d \) and an absolutely continuous measure \( \varrho_0\in \mathcal{P}_c(\mathbb{R}^d)\) satisfy Assumptions \( (A_{1,2}) \).
  Let \( \Pi(\varrho_0) \) be defined as in~\eqref{eq:Pi} for some convex compact set \( \Omega \subset \mathbb{R}^d \).
  Then, there exist $\kappa,r>0$ such that \( S_r^{\Pi}(\varrho_0) \) is \( \kappa \)-stabilizable around the reference trajectory $\mu_t$ of equation~\eqref{eq:conteq}.

  The numbers \( \kappa \), \( r \) as well as \( C \) from the definition of \( \kappa \)-stabilizability only depend on \(\Omega, T,M,\omega,\spt(\varrho_0)\).
\end{theorem}
\vspace{3pt}

A drawback of this statement is that, the starting set \( S^{\Pi}_r(\varrho_0) \) is a very small subset of the Wasserstein ball~\( \overline{\bf B}_{r}(\varrho_0) \).
In fact, one can show that it is nowhere dense in it (see e.g. \cite{PR2024}). The following result allows us to cope, to some extent, with the smallness of \( S^{\Pi}_r(\varrho_0) \).

\begin{theorem}[{\!\!\cite[Cor. 1.4]{PR2024}}]
    \label{thm:approx}
  Let all the assumptions of Theorem~\ref{thm:main} hold, \( r,\kappa>0 \) be so that \( S_r^{\Pi}(\varrho_0) \) is \( \kappa \)-stabilized and \( \alpha(r) := -r^{1+\kappa}/\log r \).
  Then, any point of the open enlargement ${\bf B}_{\alpha(r)}(S^{\Pi}_r(\varrho_0))$ can be steered by an admissible control into the ball \( {\bf B}_{C_1r^{1+\kappa}}(\mu_T) \).
  The constant \( C_1>0 \) only depends on \( T \), \( M \), \( \omega \), \( \spt(\varrho_0) \).
\end{theorem}

The interest of our contribution is three-fold: first, it finds the most natural definition of stabilization in this setting, that is the reduction of the distance rate. Second, and most importantly, it provides a completely geometric description of the stabilization result. In particular, the fact that we use geodesics of the Wasserstein space corresponds to the use of the ``main directions'' of the tangent space itself, as explained above. Third, even though geodesic directions are few in the tangent space, Theorem \ref{thm:approx} allows us to extend the result to an neighborhood of the initial state, making the stabilization property effective in applications.

\subsection*{Notation}

For convenience, we collect here basic notations used in the paper. Below, \( \mathcal{X} \) and \( \mathcal{Y} \) are arbitrary metric spaces.

\begin{supertabular}{ll}
  \( |x|\) & the Euclidean norm on \( \mathbb{R}^{d} \)\\
  \(\id\) & the identity map \(\id\colon \mathbb{R}^d\to \mathbb{R}^d\);\\
  \( C(\mathcal{X};\mathcal{Y}) \) & the space of continuous maps \( f\colon \mathcal{X}\to \mathcal{Y} \) \\
	\( \|f\|_{\infty}\)&\(\sup_{x \in \mathbb{R}^{d} }|f(x)|  \), the norm on \( C(\mathbb{R}^d;\mathbb{R}^d) \)\\
  \( \lip (f) \) & \( \sup_{x \ne y }\frac{|f(x)-f(y)|}{|x-y|}  \), for \( f\in C(\mathbb{R}^d;\mathbb{R}^d) \)\\
  \( \partial f \) & subdifferential of a convex map \( f\colon \mathbb{R}^d\to \mathbb{R} \)\\
  \( \partial A \) & boundary of a set \( A \subset \mathbb{R}^d \)\\
  \(\co (A)\) & the convex hull of \(A\subset \mathbb R^d\)\\
  \( {B}_{r}(A)\) &\(\left\{x\in \mathbb{R}^d\,\colon\, \inf_{y\in A}|x-y|< r\right\}\), for \( A \subset \mathbb{R}^d\)\\
  \( \overline{B}_{r}(A)\) &\(\left\{x\in \mathbb{R}^d\,\colon\, \inf_{y\in A}|x-y|\le r\right\} \), for \( A \subset \mathbb{R}^d \)\\
  \(\mathcal L^d\) & the Lebesgue measure on \(\mathbb R^d\)\\
  \( \spt (\mu) \) & the support of the probability measure \( \mu \)\\
  \( \sharp \) & the pushforward operator (see Section~\ref{sec:wass})\\
   \(\mathcal{P}(\mathcal{X}) \) & the space of probability measures on \( \mathcal{X} \)\\
  \(\mathcal{P}_{2}(\mathbb{R}^{d}) \) & \( \left\{\mu\in \mathcal{P}(\mathbb{R}^d)\,\colon\,\int |x|^{2}\,d\mu<\infty\right\} \)\\
  \( \mathcal{P}_{c}(\mathbb{R}^{d}) \) & \( \left\{\mu\in \mathcal{P}(\mathbb{R}^d)\,\colon\, \spt(\mu)\text{ is compact}\right\} \)\\
  \( \mathcal{W}_2 \) & the 2-Wasserstein distance on \( \mathcal{P}_{2}(\mathbb{R}^{d}) \)\\
\end{supertabular}


\section{Wasserstein space and its geodesics}
\label{sec:wass}

This section provides basic information about the Wasserstein space and discuss some properties of the intermediate optimal transport maps, related to its geodesics.

\subsection{Basic definitions and properties}
\label{subsec:prob}
Here we collect several standard results about the \emph{Wasserstein space} \( \mathcal{P}_{2}(\mathbb{R}^{d}) \), which is the space of probability measures \(\mu\) with finite second moments, i.e., such that \(\int|x|^2\,d\mu(x)<\infty\).
For a detailed treatment, see  \cite{AGS05,santambrogioOptimalTransportApplied2015,villaniTopicsOptimalTransportation2003,zbMATH05306371}.

We begin by reviewing the standard concepts of pushforward measure and transport plan.

\begin{definition}
  Let \(\mu\) be a Borel probability measure on \(\mathbb R^m\) and  \( f\colon \mathbb{R}^m\to \mathbb{R}^n \) be a Borel map.
  The probability measure \(f_\sharp\mu\) on \(\mathbb R^n\) constructed by the rule
\begin{displaymath}
  (f_{\sharp}\mu)(A):= \mu\left(f^{-1}(A)\right),\quad \text{for all Borel sets }A\subset Y,
\end{displaymath}
is called \emph{the pushforward} of \(\mu\) by \(f\).
\end{definition}

\begin{definition}
Let \(\mu\) and  \(\nu\) be any Borel probability measures on \(\mathbb R^d\).
A \emph{transport plan} \( \bm \gamma \) between \(\mu\) and \(\nu\) is a probability measure on \( \mathbb{R}^{d}\times \mathbb{R}^{d} \) whose projections on the first and the second factor are
\( \mu \) and \( \nu \), respectively. In other words, 
\(
  \pi^{1}_{\sharp}\bm\gamma = \mu,\quad \pi^{2}_{\sharp}\bm\gamma=\nu,
\)
where \( \pi^{1}, \pi^{2} \) are the projections on the factors.
The set of all transport plans  between \( \mu \) and \( \nu \) is denoted by  \( \Gamma(\mu,\nu) \).
\end{definition}

\begin{definition}
  A transport plan \( \bm\gamma\in \Gamma(\mu,\nu) \) is said to be generated by a \emph{transport map} \( \mathcal{T}\colon \mathbb{R}^d\to \mathbb{R}^d \) if it takes the form \( \bm\gamma = (\id,\mathcal{T})_{\sharp}\mu \).
  It is evident that in this case \( \mathcal{T}_{\sharp}\mu=\nu \).
\end{definition}
\begin{definition}
  Given a pair of measures \( \mu,\nu\in \mathcal{P}_2(\mathbb{R}^d) \), the corresponding \emph{optimal transportation problem} is the minimization problem
  \(
  \displaystyle \inf_{\bm\gamma\in\Gamma(\mu,\nu)}\int|x-y|^{2}\,d \bm\gamma(x,y).
  \)
  Any transport plan \( \bm\gamma\in \Gamma(\mu,\nu) \) that solves it is called \emph{optimal}.
 \end{definition}

 Optimal plans always exist~\cite[Thm. 1.3]{villaniTopicsOptimalTransportation2003} and admit a simple characterization given below. 
 \begin{theorem}[{\!\!\cite[Thm. 2.12]{villaniTopicsOptimalTransportation2003}}]
   \!\!\mbox{If\! \( \mu,\nu\!\in\! \mathcal{P}_2(\mathbb{R}^d) \), then}
   \begin{enumerate}
     \item \( \bm\gamma\in \Gamma(\mu,\nu) \) is optimal if and only if the support \( \spt (\bm\gamma) \) belongs to the subdifferential \( \partial \phi \) of a convex lower semicontinuous function \( \phi\colon \mathbb{R}^d\to \mathbb{R}\cup\{+\infty\} \);

     \item if \( \mu \) is absolutely continuous, then there exists a unique optimal plan \( \bm\gamma\in \Gamma(\mu,\nu) \), which is \( \bm\gamma=(\id,\nabla\phi)_{\sharp}\mu \), where \( \nabla\phi \) is the unique (i.e., uniquely determined \( \mu \)-almost everywhere) gradient of a convex function \( \phi \).
   \end{enumerate}
 \end{theorem}
 \vspace{3pt}

One can use optimal transportation to define the so called Wasserstein distance on \( \mathcal{P}_{2}(\mathbb{R}^{d}) \), turning it into a \emph{complete separable metric space}~\cite[Thm.~6.18]{zbMATH05306371}.

\begin{definition} The \emph{Wasserstein distance} is
\begin{equation}
  \label{eq:wasserstein}
  \mathcal{W}_2(\mu,\nu) := \left(\inf_{\Pi\in\Gamma(\mu,\nu)}\int|x-y|^{2}\,d \Pi(x,y)\right)^{1/2}.
\end{equation}
\end{definition}
\vspace{5pt}

Before proceeding further, we recall useful definitions from metric geometry~\cite{zbMATH01626771}.
\begin{definition}
  Let \( \gamma\colon [a,b]\to \mathcal X \) be a continuous curve on a metric space \((\mathcal{X},d)\).
  The \emph{length} of \( \varrho \) is defined by
\begin{displaymath}
  L(\gamma) := \sup \sum_{i=1}^{N}d(\gamma_{t_{i-1}},\gamma_{t_{i}}),
\end{displaymath}
where the supremum is taken among all finite partitions \( a=t_{0}\leq t_{1}\leq \cdots\leq
t_{N}=b \) of the interval \( [a,b] \).
\end{definition}
\vspace{3pt}

\begin{definition}
  A \emph{length space} is a metric space \( (\mathcal{X},d) \) such that
  \[
    d(x,y) = \inf_{\gamma \in C([0,1]; \mathcal{X}) }\left\{ L(\gamma)\;\colon\; \gamma_{0}=x,\;\gamma_{1}=y \right\}.
  \]
  If the infimum above is always attained, then \( \mathcal{X} \) is said to be a \emph{strictly intrinsic length space}, or a \emph{geodesic space}.
\end{definition}
\vspace{3pt}
\begin{definition}
  Let \( (\mathcal{X},d) \) be a geodesic space.
  A curve \( \gamma\in C([0,1];\mathcal{X}) \) joining \( x, y\in \mathcal{X} \) with \( d(x,y)=L(\gamma) \) is called a \emph{minimal geodesic} between \( x \) and \( y \).
\end{definition}
\vspace{3pt}

It is known~\cite[Sec. 2.5]{zbMATH01626771} that any minimal geodesic admits a constant speed parametrization, which satisfies
\begin{displaymath}
  d(\gamma_{t},\gamma_{s}) = |t-s|d(x,y), \quad s,t\in [0,1].
\end{displaymath}
In what follows, by saying that \( \gamma_t \) is a \emph{geodesic} joining \( x \) and \(y\) we mean that \( \gamma_t \) is a minimal geodesic from \( x\) to \( y \) parametrized in this way.

Finally, we can characterize the Wasserstein space \( (\mathcal{P}_2(\mathbb{R}^d),\mathcal{W}_2) \) as a geodesic space.

\begin{theorem}[{\cite[Cor. 2.9]{gigliGeometrySpaceProbability2004}}]
  The Wasserstein space \( (\mathcal{P}_2(\mathbb{R}^d),\mathcal{W}_2) \) is a geodesic space.
Any geodesic \( \varrho_t \) joining \( \varrho_0, \varrho_1 \in \mathcal{P}_2(\mathbb{R}^d)\) takes the form
\( \rho_{t} = \left((1-t)\pi^{1}+t \pi^{2}\right)_{\sharp}\bm\gamma \), where \( \bm\gamma \) is an optimal transport plan between
\( \varrho_0 \) and \( \varrho_1 \).
In particular, if the optimal plan is unique, the geodesic is  unique as well.
\end{theorem}


\subsection{Prolongation of geodesics}

Consider the starting set \( S^{\Pi}_r(\varrho_0) \), where \( \Pi \) is defined by~\eqref{eq:Pi} and \( r\in(0,1) \).
Observe that any point \( \mu\in S^{\Pi}_r(\varrho_0) \), except \( \varrho_0 \), is an \emph{intermediate point} of some geodesic \( \gamma_{\varrho_0,\varrho_1} \), where \( \varrho_1\in \mathcal{P}_{ac}(\Omega) \).
Equivalently, we may say that \( \mu \) is an intermediate point of a geodesic starting at \( \varrho_0 \) if the geodesic \( \gamma_{\varrho_0,\mu} \) can be prolonged.
In this section, we give a precise notion of prolongation and characterize intermediate points.

\begin{definition}
Let \( \varrho\colon [0,1]\to \mathcal{P}_2(\mathbb{R}^d) \) be a geodesic and \( a,b\in [0,1] \). Assume that the optimal transport plan between \( \varrho_a \) and \( \varrho_b \) comes from a transport map: we denote the map by \( \mathcal{T}_{a,b} \).
\end{definition}

The following result is a corollary of~\cite[Prop. 2.11]{gigliGeometrySpaceProbability2004}.

\begin{lemma}
  \label{lem:intermediate}
  Let \( \varrho\colon [0,1]\to \mathcal{P}_2(\mathbb{R}^d) \) be a geodesic and \( a,b\in (0,1)  \) be such that \( a<b \).
  Then the optimal transport plans between \( \varrho_a \) and \( \varrho_b \) and between \( \varrho_b \) and \( \varrho_a \) are unique. 
  These plans are realized by transport maps \( \mathcal{T}_{a,b} \) and \( \mathcal{T}_{b,a} \).
  Moreover, the maps \( \mathcal{T}_{a,b} \), \( \mathcal{T}_{b,a} \) are Lipschitz continuous, with
  \[
    \lip(\mathcal{T}_{b,a})\le \frac{1-a}{1-b},\quad
    \lip(\mathcal{T}_{a,b})\le \frac{b}{a}.
  \]
\end{lemma}
\begin{proof}
  Let \( \mu\colon [0,1]\to \mathcal{P}_2(\mathbb{R}^d) \) be an arbitrary geodesic.
  According to~\cite[Prop. 2.11]{gigliGeometrySpaceProbability2004}, for any \( s\in (0,1)\), optimal plans between \( \mu_s \), \( \mu_0 \) and between \( \mu_s \), \( \mu_{1} \) are unique and come from maps \( \mathcal{T}_{s,0} \) and \( \mathcal{T}_{s,1} \) such that
  \begin{equation}
    \label{eq:gigli}
    \lip(\mathcal{T}_{s,0})\le \frac{1}{1-s},\quad \lip(\mathcal{T}_{s,1})\le \frac{1}{s}.
\end{equation}



  Let \( \mu_t := \varrho_{a+(1-a)t} \), \( t\in [0,1] \), so that
    \[
    \mu_0 = \varrho_a,\quad \mu_{\frac{b-a}{1-a}} = \varrho_{b},\quad \mu_1 = \varrho_{1}.
    \]
    The curve \( \mu_t \) is clearly a geodesic.
    By applying~\eqref{eq:gigli} for \( \mu_t \) and \( s =\frac{b-a}{1-a}\), we obtain \( \lip(\mathcal{T}_{b,a})\le \frac{1-a}{1-b} \).
    To deal with the reverse map, take \( \mu_{t} := \varrho_{tb} \); then
    \[
    \mu_0 = \varrho_0,\quad \mu_{\frac{a}{b}} = \varrho_{a},\quad \mu_1 = \varrho_{b},
    \]
    and we get the second inequality.
\end{proof}

\begin{definition}
  Let \( \varrho\colon [0,1]\to \mathcal{P}_2(\mathbb{R}^d) \) be a geodesic.
  We say that \( \varrho \) can be \emph{prolonged forward} if there exists another geodesic \( \mu\colon [0,1]\to \mathcal{P}_2(\mathbb{R}^d) \)
  and a constant \( a\in(0,1) \) such that \( \varrho_t = \mu_{at} \) for all \( t\in [0,1] \).
Similarly, \(\varrho\) can be \emph{prolonged backward} if \(\varrho_{1-t}\) can be prolong forward.
\end{definition}
\vspace{5pt}

\begin{theorem}[Prolongation of geodesics]
  \label{thm:prolong}
  A geodesic \( \varrho\) can be prolonged forward if and only if the optimal transport plan between \( \varrho_1 \) and \( \varrho_s \) comes from a Lipschitz transport map \( \mathcal{T}_{1,s}\), for each \( s\in(0,1) \). Similarly for the backward propagation.
\end{theorem}
\begin{proof}
  We only consider forward prolongation.

  {\bf \noindent $\Rightarrow$)} Let \( \mu \) be a forward prolongation of \( \varrho \).
  Then \( \varrho_1 \) and \( \varrho_s \) are internal points of \( \mu_t \).
  By Lemma~\ref{lem:intermediate}, there exists a unique optimal plan between them, and it comes from a Lipschitz transport map.

 {\bf \noindent $\Leftarrow$)} Suppose that a unique optimal transport plan between \( \varrho_1 \) and \( \varrho_s \) is realized by a Lipschitz transport map \( \mathcal{T}_{1,s} \).
  Thanks to~\cite[Prop. 2.11]{gigliGeometrySpaceProbability2004}, there exists \( \mathcal{T}_{s,1} = \mathcal{T}_{1,s}^{-1} \) and, being an optimal transport map, it takes the form \( \mathcal{T}_{s,1}= \nabla\phi \) for a convex function \( \phi \).

  Let \( \phi^{*} \) denotes the convex conjugate of \( \phi \).
  We have \(\nabla\phi^{*} = \left(\nabla\phi\right)^{-1}\), by~\cite[Cor. 23.5.1]{rockafellarConvexAnalysis1972}, implying that \( \nabla\phi^{*} \) is Lipschitz with some Lipschitz constant \( \ell^{-1}>0 \).
  Now, thanks to~\cite[Prop. 12.60]{rockafellarVariationalAnalysis1998}, the dual \( (\phi^{*})^{*}=\phi \) is strongly convex with constant \( \ell \), i.e.,
\[
\left\langle \nabla\phi(x)- \nabla\phi(y),x-y \right\rangle \ge \frac{\ell}{2}|x-y|^2.
\]

Consider the following function:
\[
  \phi_{\varepsilon}(x) := -\frac{\varepsilon}{2} |x|^2 + (1+\varepsilon) \phi(x),\quad \varepsilon>0.
\]
Then
\begin{align}
  \label{eq:dphieps}
    \big\langle \nabla&\phi_{\varepsilon}(x)-\nabla\phi_{\varepsilon}(y),x-y\big\rangle\notag\\
    &= (1+\varepsilon)\left\langle \nabla\phi_{\varepsilon}(x)-\nabla\phi_{\varepsilon}(y),x-y\right\rangle -\varepsilon|x-y|^2 \notag\\
    &\ge \left[\ell/2+ \varepsilon((\ell/2)-1)\right]\,|x-y|^{2}.
  \end{align}
  Choose \( \sigma=1 \) when \( \ell\ge 2 \) or any positive number below \( \ell/(2-\ell)\) when \( \ell<2 \).
  Then, by~\eqref{eq:dphieps}, it holds
  \[
    \left\langle \nabla\phi_{\varepsilon}(x)-\nabla\phi_{\varepsilon}(y),x-y\right\rangle
    \ge \sigma\,|x-y|^{2},
  \]
  i.e., \( \phi_{\sigma} \) is strongly convex.
  As a consequence, \( \nabla\phi_{\sigma} \) is an optimal transport map.
  The corresponding geodesic \( \mu_t \), which connects \( \mu_0=\varrho_s \) with \( \mu_1= (\nabla\phi_{\sigma})_{\sharp}\varrho_s \), takes the form
  \begin{align*}
    \mu_t &= \left((1-t)\id + t\nabla\phi_{\sigma}\right)_{\sharp}\varrho_s\\
    &= \left((1-t(1+\sigma))\id + t(1+\sigma)\nabla\phi\right)_{\sharp}\varrho_s.
  \end{align*}
  Therefore, \( \mu_{t} = \varrho_{(1-s)(1+\sigma)t+s} \) for all \( t\in \left[0,(1+\sigma)^{-1}\right]\).

  Note that \( \mu_t \) defined above depends on \( s \).
  By taking \( s=1/n \), \( n\in \mathbb{N} \), we obtain a sequence of geodesics \( \mu^n\colon [0,1]\to \mathcal{P}_2(\mathbb{R}^d) \).
  The pointwise limit of the maps \( \mu^n \) is again a geodesic due to~\cite[Prop. 2.5.17]{zbMATH01626771}.
  This geodesic clearly contains \( \varrho\).
\end{proof}

\begin{remark}
There is only one optimal transport plan between \(\varrho_1\) and \(\varrho_s\) when \(s\in (0,1)\) (see~\cite[Prop. 2.11]{gigliGeometrySpaceProbability2004}).
However, such a plan may not be generated by a transport map, as in the case \(\varrho_1 = \delta_0\), \(\varrho_0 = \frac{1}{2}\delta_{-1}+\frac{1}{2}\delta_1\).
\end{remark}

\section{Main result}
\label{sec:MainResult}

In this section, we prove Theorem~\ref{thm:main}.

We begin by defining one more perturbation class, that will appear useful in the future.

\begin{definition}[Regular perturbations]
  \label{def:reg_pert}
Fix a measure \( \varrho_0\in \mathcal{P}_c(\mathbb{R}^d) \), a number \( \varepsilon_0\in (0,1) \) and a nondecreasing function \( L\colon \mathbb{R}_+\to \mathbb{R}_+ \).
  We say that \( \varrho\colon [0,1]\to \mathcal{P}_c(\mathbb{R}^{d}) \) belongs to the class \( \Pi^{\varepsilon_0,L}(\varrho_0) \) of \emph{regular perturbations} of \( \varrho_0 \) if there exists a map \( \Psi\colon [0,\varepsilon_0] \times \mathbb{R}^d\to \mathbb{R}^d \)
such that
\begin{enumerate}[(a)]
  \item for all \( \varepsilon\in [0,\varepsilon_0] \), \( \Psi_{\varepsilon} \) is locally Lipschitz: for all $x,y\in \overline{B}_r(0),  r>0$ it holds
        \begin{equation}
          \label{eq:Psi1}
        \left|\Psi_{\varepsilon}(x)-\Psi_{\varepsilon}(y)\right|\le L(r)|x-y|\\ ;
        \end{equation}
  \item for all \( \varepsilon\in [0,\varepsilon_0] \), \( \Psi_{\varepsilon} \) is locally ``close to the identity'':
        \begin{equation}
          \label{eq:Psi2}
         \left|\Psi_{\varepsilon}(x)-x\right| \le \varepsilon L(r)\quad \forall\, x\in \overline{B}_r(0),\, r>0;
        \end{equation}
  \item for all \( \varepsilon\in [0,\varepsilon_0] \) and \( x\in \mathbb{R}^d \), the matrix \( D\Psi_{\varepsilon}(x) \) is positive semi-definite, if it exists;
  \item \(\varrho_0 = \Psi_{\varepsilon\sharp}\varrho_{\varepsilon}\) for all \( \varepsilon\in [0,\varepsilon_{0}] \).
\end{enumerate}
\end{definition}
\vspace{3pt}

While the definition seems technical, it simply lists minimal properties of a perturbation that we require to prove the following theorem.

\begin{theorem}[{\!\!\cite[Thm. 4.1]{PR2024}}]
  \label{thm:main2}
  Let the assumptions \( (A_{1,2}) \) hold, a constant \( \varepsilon_0>0 \), a nondecreasing function \( L\colon \mathbb{R}_+\to \mathbb{R}_+\) and a measure \( \varrho_0\in \mathcal{P}_c(\mathbb{R}^d) \) be fixed.
  Denote by \( S^{\varepsilon_0,L}_r(\varrho_0) \) the starting set corresponding to the class of regular perturbation \( \Pi^{\varepsilon_0,L}(\varrho_0) \).
  Then, there exist $\kappa>0$ and \( r \in (0,\varepsilon_0) \) such that \( S^{\varepsilon_0,L}_r(\varrho_0) \) is \( \kappa \)-stabilizable around the reference trajectory $\mu_t$ of equation~\eqref{eq:conteq}.

  The numbers \( \kappa \), \( r \) as well as \( C \) from the definition of \( \kappa \)-stabilizability only depend on \( \varepsilon_0,L,T,M,\omega,\spt(\varrho_0)\).
\end{theorem}
\vspace{3pt}

\begin{remark}
  Observe that \cite[Thm. 4.1]{PR2024} was proved under slightly more restrictive hypothesis: the notion of \emph{regular perturbation} employed in~\cite{PR2024} assumed that \( L \) is a positive constant and inequalities~\eqref{eq:Psi1} and~\eqref{eq:Psi2} hold globally on \( \mathbb{R}^d \).
  However, it is easy to see that the theorem also holds if regular perturbations are defined as in the present paper.
  Indeed, suppose that \( L \) is constant and \( r \) is so that \( S^{\varepsilon_0,L}_r(\varrho_0) \) is \( \kappa \)-stabilized.
  Since \( \varrho_0 \) is compactly supported and \( V \) is bounded, there exists \( R>0 \) such that the reference trajectory satisfies \( \spt(\mu_t) \subset \overline{B}_R(0) \) for all \( t\in [0,T] \).

  Now, if \( \varrho\in S^{\varepsilon_0,L}_r(\varrho_0) \), then the corresponding stabilizing trajectory \( \mu_t^u \) satisfies \( \spt(\mu_t^u) \subset \overline{B}_{R+C_1r}(0)  \), where \( C_1>0 \) does not depend on \( \varrho \) but only on \( L \).
  This is a consequence of the estimates obtained in the proof of \cite[Thm. 4.1]{PR2024}.
  In other words, it suffices to require~\eqref{eq:Psi1} and~\eqref{eq:Psi2} to hold only inside \( \overline{B}_{R+C_1r}(0) \), which leads to Definition~\ref{def:reg_pert}.
\end{remark}
\vspace{3pt}

First, we show that all geodesics from \( \Pi(\varrho_0) \) are regular perturbations when \( \spt(\varrho_0) \) is convex.

\begin{lemma}
  \label{lem:reg}
  Let \( \Omega \subset \mathbb{R}^d \) be a compact convex set, \( \varrho_0\in \mathcal{P}_{ac}(\Omega) \)  and \( \Pi(\varrho_0) \) be defined by~\eqref{eq:Pi}.
  Assume in addition that \( \spt(\varrho_0) \) is convex.
  Then, \( \Pi(\varrho_0) \subset \Pi^{1/2,L}(\varrho_0)\), where \( L(r):= \max\{4R + 2|r|,2\} \) and \( R:=\diam(\Omega)\).
\end{lemma}
\begin{proof}
  Fix some \( \varrho_1\in \mathcal{P}_{ac}(\Omega) \).
  Denote by \( t\mapsto \varrho_t \) the unique geodesic joining \(\varrho_0\) and \( \varrho_1 \) and by \( \mathcal{T}_{s,t} \) the corresponding intermediate transport map.
  Note that \( \mathcal{T}_{0,\varepsilon} \) exists for each \( \varepsilon\in [0,1] \), since \( \varrho_0 \) is absolutely continuous.

  Since \( \mathcal{T}_{0,1} \) is an optimal transport map, it takes the form \( \mathcal{T}_{0,1}(x) = \partial\phi(x) \) for \( \varrho_0 \)-a.e. \( x \), where \( \phi \colon \mathbb{R}^d\to \mathbb{R}\cup\{+\infty\} \) is a proper convex lower semicontinuous function and $\partial$ denotes the subdifferential.
  Since, by definition, \( \mathcal{T}_{0,\varepsilon} = (1-\varepsilon)\id +\varepsilon\mathcal{T}_{0,1} \), we have
  \( \mathcal{T}_{\varepsilon,0} = \frac{1}{1-\varepsilon}\left(\id +\frac{\varepsilon}{1-\varepsilon}\partial\phi\right)^{-1} = \frac{1}{1-\varepsilon}J_{\frac{\varepsilon}{1-\varepsilon}} \),
  where \( J_{\lambda}:= (\id +\lambda \partial\phi)^{-1} \) is called the \emph{resolvent} of \( \partial\phi \).
  We can also define the map \( (\partial\phi)_{\lambda}:= \frac{1}{\lambda}(\id-J_{\lambda})\) called the \emph{Yosida approximation} of \( \partial\phi \).

  Note that \( \partial\phi \) is a maximal monotone map as the subdifferential of a proper convex lower semicontinuous function.
  Thus, by~\cite[Thm. 2, p. 144]{zbMATH03855514}, the resolvent \( J_{\lambda}\colon \mathbb{R}^d\to \mathbb{R}^d \) is single-valued and nonexpansive, for each \( \lambda>0 \).
Moreover,
\begin{align}
  \left|J_{\lambda}(x)-x\right|
  &= \lambda \left|(\partial\phi)_{\lambda}(x)\right|\notag\\
  &\le \lambda \max\left\{|p|\;\colon\; p\in \partial\phi(J_{\lambda}(x)) \right\},
  \label{eq:terrible}
\end{align}
for all \( x\in \mathbb{R}^d \).
The last inequality is due to the inclusion \( (\partial\phi)_{\lambda}(x)\in \partial\phi(J_{\lambda}(x)) \), which holds for all \( x\in \mathbb{R}^d \).

We state that \( \spt(\varrho_{0})\subset \dom (\phi) \).
Indeed, \( \spt(\varrho_{0})\subset \overline{\dom (\phi)} \), by construction~\cite{mccannExistenceUniquenessMonotone1995}.
Hence we just need to check that any \( x\in \partial\spt(\varrho_0)\cap \partial\dom(\phi) \) belongs to \( \dom(\phi) \).
Let \( R =\diam(\Omega)\).
Since \( \varrho_1 = (\nabla\phi)_{\sharp}\varrho_0 \) (see~\cite{mccannExistenceUniquenessMonotone1995}) and \( \spt(\varrho_1)\subset \Omega \), we conclude that
\[
  A_{R} := \{x\in \dom(\nabla\phi)\;\colon\; |\nabla\phi(x)|\le R\}
\]
is a set of full measure \( \varrho_0 \), i.e., \( \varrho_0(A_{L})=1 \).
Note that \( \varrho_0\left(B_{\varepsilon}(x)\cap \spt(\varrho_0)\right)>0 \), for any \( \varepsilon>0 \).
Thus, we can find a sequence \( x_n\to x \) such that \( x_n\in A_R \), for all \( n\in \mathbb{N} \).
Since \( |\nabla\phi(x_n)|\le R \), for all \( n\in \mathbb{N} \), the sequence \( \phi(x_n) \) only admit finite limits.
In particular, \( \liminf_{n\to\infty} \phi(x_n)<+\infty \).
Hence \( \phi(x)< +\infty \), by the lower semicontinuity of \( \phi \), and therefore \( x\in \dom(\phi) \).

Being a convex function, \( \phi \) is Lipschitz continuous on the compact set \( \spt(\varrho_0) \).
Since \( A_R \) has full measure, we conclude that \( \lip(\phi)\le R \) on \( \spt(\varrho_0) \).
Hence, \(\phi\) can be extended~\cite[Thm. 1]{cobzasNormpreservingExtensionConvex1978} from the convex set \( \spt(\varrho_0) \) to an \( R \)-Lipschitz convex function defined on the whole space \( \mathbb{R}^d \).
Thus, without loss of generality, we may assume that \( \dom(\phi)=\mathbb{R}^d \) and \( \partial \phi(x)\subset \overline{B}_R(0) \) for all \( x\in \mathbb{R}^d \).
Now,~\eqref{eq:terrible} implies that
\[
|J_{\lambda}(x)-x|\le \lambda R\quad \forall x\in \mathbb{R}^d\quad \forall \lambda\in \mathbb{R}_+.
\]


Now, for any \( r>0 \) and \( x\in \overline{B}_r(0) \), we have
  \begin{align*}
    \left|\mathcal{T}_{\varepsilon,0}(x)-x\right|
    &= \left|\frac{1}{1-\varepsilon}J_{\frac{\varepsilon}{1-\varepsilon}}(x) - x\right|\\
    &\le \frac{1}{1-\varepsilon}\left|J_{\frac{\varepsilon}{1-\varepsilon}}(x) - x\right| + \left|\frac{x}{1-\varepsilon} - \frac{(1-\varepsilon)x}{1-\varepsilon}\right|\\
    &\le \frac{\varepsilon}{(1-\varepsilon)^2}R + \frac{\varepsilon}{1-\varepsilon}r.
  \end{align*}
Thus, if we let \( \Psi_{\varepsilon}:=\mathcal{T}_{\varepsilon,0} \), we conclude that~\eqref{eq:Psi2} holds for \( \varepsilon_0=1/2 \) and \( L(r)=4R+2|r| \).
Moreover, \( \lip(\mathcal{T}_{\varepsilon,0})\le 2 \), since \( J_{\lambda} \) is nonexpansive.
Property (c) from Definition~\ref{def:reg_pert} is true thanks to Rademacher’s Theorem~\cite[Thm. 6.6]{zbMATH06413063}.
Finally, (d) holds by the definition of \( \mathcal{T}_{\varepsilon,0} \).
The proof is complete.
\end{proof}

Take \( \varrho,\mu\in \mathcal{P}_c(\Omega) \) and denote by \( \mathcal{T}_{\varrho}^{\mu} \) the unique optimal transport map between \( \varrho \) and \( \mu \).
Consider the Lebesgue space \( L^2_{\varrho}(\Omega;\mathbb{R}^d) \) consisting of square integrable functions (equivalence classes) with respect to the measure \( \varrho \) and denote its norm by \( \|\cdot\|_{\varrho} \). We have the following convergence result, that will be used in the proof of the main theorem.

\begin{lemma}
  \label{lem:cont}
If \( \mu_k\to \mu \) in \( \mathcal{P}(\Omega) \), then \( \mathcal{T}_{\varrho}^{\mu_k}\to \mathcal{T}_{\varrho}^{\mu} \) in \( L^2_\varrho(\Omega;\mathbb{R}^d) \).
\end{lemma}
\begin{proof}
  Note that \( \mu_k\to\mu \) in \( \mathcal{P}(\Omega) \), \( (\id,T_{\varrho}^{\mu})_{\sharp}\varrho \) is the unique optimal transport plan between \( \varrho \) and \(\mu\), and \( T_{\varrho\sharp}^{\mu_k}\varrho = \mu_k \).
  Moreover,
\[
\limsup_{k \to \infty} \int|T_{\varrho}^{\mu_k}-\id|^2\,d\varrho = \limsup_{k \to \infty} \mathcal{W}^{2}_2(\varrho,\mu_k)= \mathcal{W}^{2}_2(\varrho,\mu).
\]
Thus, by~\cite[Prop. 2.27]{gigliGeometrySpaceProbability2004}, \( (T_{\varrho}^{\mu_n},\varrho) \) converges strongly to \( (T_{\varrho}^{\mu},\varrho) \) in the sense of~\cite[Def.~2.20]{gigliGeometrySpaceProbability2004}.
By this definition, \( T_{\varrho}^{\mu_k}\to T_{\varrho}^{\mu} \) weakly in \( L^2_{\varrho}(\Omega;\mathbb{R}^d) \) and \( \|T_{\varrho}^{\mu_k}\|_{\varrho}\to \|T_{\varrho}^{\mu}\|_{\varrho} \), that is, \( T_{\varrho}^{\mu_k}\to T_{\varrho}^{\mu} \) strongly in \( L_{\varrho}^2(\Omega;\mathbb{R}^d) \).
\end{proof}

\begin{proofof}{Theorem~\ref{thm:main}}
  By combining Theorem~\ref{thm:main2} and Lemma~\ref{lem:reg}, we conclude that Theorem~\ref{thm:main} holds if \( \spt(\varrho_0) \) is convex.
  In other words, in this case any point from the starting set \( S^{\Pi}_r(\varrho_0) \) can be steered to the ball \( {\bf B}_{Cr^{1+k}}(\mu_T) \).

  Now, suppose that \( \spt(\varrho_0) \) is nonconvex.
  Clearly, in any neighborhood of \(\varrho_0\) we may find \( \tilde\varrho_0 \) such that \( \spt(\tilde\varrho_0)=\co(\spt(\varrho_0)) \).
  By Theorem~\ref{thm:main2} and Lemma~\ref{lem:reg}, the starting set \( S^{\Pi}_r(\tilde\varrho_0) \) is \( \kappa \)-stabilized, for some \( \kappa,r>0 \).
  Moreover, these constants does not depend on the choice of \( \tilde \varrho_0 \) but only on \( \spt(\tilde\varrho_0)=\co(\spt(\varrho_0)) \).
  Below, we assume that \( \kappa,r \) are fixed.

  Take some \( \varrho_1\in \mathcal{P}_{ac}(\Omega)  \) and consider the geodesics \( \varrho_t \) and \( \tilde \varrho_t \) joining \( \varrho_1 \) with \(\varrho_0\) and \( \tilde\varrho_0  \), respectively.
It clearly holds
\begin{equation}
  \label{eq:geod}
\mathcal{W}_2^2(\varrho_t,\tilde \varrho_t) \!\le\! \int\!|T_{\varrho_1}^{\varrho_t} - T_{\varrho_1}^{\tilde \varrho_t}|^2d \varrho_1\!\le\! \int\!|T_{\varrho_1}^{\varrho_0} - T_{\varrho_1}^{\tilde \varrho_0}|^2d \varrho_1,
\end{equation}
for all \( t\in [0,1] \).
By Lemma~\ref{lem:cont}, we may always choose \( \tilde \varrho_0 \) so that the right-hand side in~\eqref{eq:geod} is arbitrarily small.
Thus, the measures \( \varrho_t \), \( t\in [0,r] \) lie inside the enlargement \( {\bf B}_{\alpha(r)}(S^{\Pi}_r(\tilde \varrho_{0})) \), where \( \alpha(r)=-r^{1+\kappa}/\log r \).
By Theorem~\ref{thm:approx}, these measures can be steered into the ball \( {\bf B}_{C_{1}r^{1+\kappa}}(\mu_T) \), for some \( C_1>0 \) only depending on \( T \), \( M \), \( \omega \), \( \co(\spt(\varrho_0)) \).
The proof is complete.
\end{proofof}

We do not give a formal proof of Theorem \ref{thm:approx}: one can derive it from Theorem \ref{thm:main} by following the proof of \cite[Cor. 1.4]{PR2024}

We conclude this section by observing that the ball \( {\bf B}_{r}(\varrho_0) \) cannot be \( \kappa \)-stabilized no matter which \( \kappa,r>0 \) we choose.

 \begin{example}[the Wasserstein ball cannot be stabilized]
   \label{ex:nostabilization}
Take \( d=1 \), \(V=0\), \(\omega=\mathbb R\), \(\varrho_0=\bm1_{[0,1]}\mathcal L^1\),
      \[
        \varepsilon_{n} := \frac{1}{2^{n-1}},\quad
       \varrho_{\varepsilon_n} := \frac{1}{2^{n-1}}\sum_{k=1}^{2^{n-1}}\delta_{\frac{2k-1}{2^{n}}},\quad n\ge 1.
     \]
     Note that \( \varrho_{\varepsilon_n} \) has \( 2^{n-1} \) atoms of mass \( 1/2^{n-1} \), and the distance between the closest atoms is \( 1/2^{n-1} \).
     Since the optimal transport map between \( \bm 1_{[a,b]} \mathcal{L}^1 \) and \( (b-a)\delta_{\frac{b-a}{2}} \) is \( \mathcal{T}\equiv \frac{b-a}{2} \), it holds
     \[
       \mathcal{W}_2(\varrho_{\varepsilon_n},\varrho_0) =\left(\sum_{k=1}^{2^{n-1}}\int_{\frac{2k-2}{2^{n}}}^{\frac{2k}{2^n}}\left|x-\frac{2k-1}{2^{n}}\right|^{2}\!d x\right)^{1/2}\!=\frac{\varepsilon_n}{2\sqrt{3}}.
     \]
    Lipschitz controls can only steer \(\varrho_{\varepsilon_{n}}\) to a measure \(\vartheta_n\) with the same atoms, but located in different positions.
    Such measures \( \vartheta_n \) satisfy \(\mathcal W_2(\varrho_0,\vartheta_n)\ge \frac{\varepsilon_n}{2\sqrt 3}\).
    Therefore, for any \(\varepsilon>0\) the ball \( {\bf B}_{\varepsilon}(\varrho_0) \) contains measures \( \varrho_{\varepsilon_n} \) that cannot be pushed towards \( \varrho_0 \) by a Lipschitz control.
    Thus, for any \( \varepsilon>0 \) and \( \kappa>0 \), the set \( {\bf B}_{\varepsilon}(\varrho_{0}) \) cannot be \( \kappa \)-stabilized.
  \end{example}


\bibliographystyle{IEEEtran}
\bibliography{references.bib}

\end{document}